\documentclass[12pt]{article}
\usepackage[final]{epsfig}
\usepackage{graphics}
\usepackage{amsmath}
\usepackage{amsfonts}
\usepackage{latexsym}
\usepackage{amssymb}
\usepackage{graphicx}
\usepackage{url}
\usepackage{epstopdf}
\usepackage{hyperref}
\usepackage{xcolor}
\usepackage{comment}

\newtheorem{lemma}{Lemma}[section]

\newtheorem{proposition}[lemma]{Proposition}
\newtheorem{remark}[lemma]{Remark}
\newtheorem{example}[lemma]{Example}
\newtheorem{theorem}{Theorem}

\newcommand{\proofend}{$\Box$\bigskip}
\newcommand{\C}{{\mathbb C}}

\newcommand{\R}{{\mathbb R}}

\def\proof{\paragraph{Proof.}}

\title{Outer symplectic billiard map at infinity}

\author{Peter Albers\footnote{
Institute for Mathematics,
Heidelberg University,
69120 Heidelberg,
Germany;
palbers@mathi.uni-heidelberg.de}
 \and 
Ana Chavez Caliz
\footnote{
Institute for Mathematics,
Heidelberg University,
69120 Heidelberg,
Germany;
anachavezcaliz@gmail.com}
\and
 Serge Tabachnikov\footnote{
Department of Mathematics,
Pennsylvania State University,
University Park, PA 16802,
USA;
tabachni@math.psu.edu}
} 

\date{}

\begin{document}

\maketitle

\begin{abstract}
We show that the second iteration $T^2$ of the outer symplectic billiard map with respect to a convex domain $M$ in a symplectic vector space is approximated by an explicit Hamiltonian flow for points far away from $M$. 

More precisely, denote by $N$ the symplectic polar dual of the symmetrization $M\ominus M$ of $M$. If we write $N$ as the unit level set of a 1-homogeneous function $H$, then the difference between $T^2$ and the time-2-Hamiltonian flow of $H$ applied to a point $x$ is smaller than $c/|x|$ for some constant $c$ depending only on $M$.

Moreover, we show that if an orbit escapes to infinity, then its distance to the origin grows not faster than $\sqrt{k}$ in the number of iterations. Finally, we prove that a $k$-periodic orbit needs to be close, in terms of $k$, to $M$.
\end{abstract}


\section{Introduction} \label{sect:intro}

The outer billiard map is an area-preserving transformation of the exterior of a strictly convex plane domain depicted in Figure \ref{map}. 
\begin{figure}[hbtp] 
\centering
\includegraphics[width=2.5in]{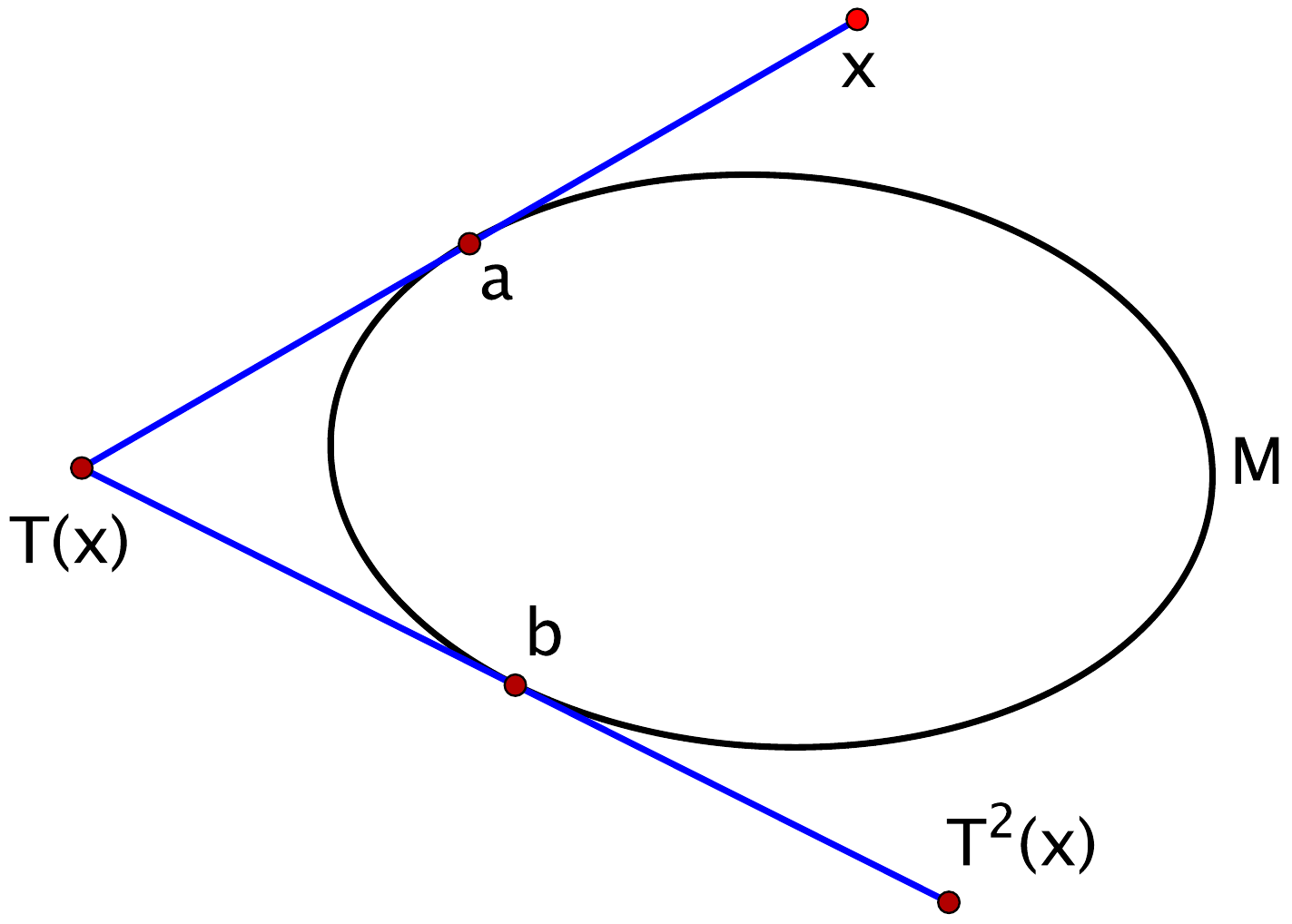}
\caption{Outer  billiard map $T$ about a convex curve $M$: $|x\, a|=|a\, T(x)|$,  $|T(x)\, b|=|b\, T^2(x)|$.}
\label{map}
\end{figure}
Introduced by B. Neumann in 1959, its study was put forward by J. Moser in the 1970s, and it has remained an active area of research since then, see \cite{Ha} for the history of this subject. 

Moser tackled the problem of whether the orbits of the outer billiard map could escape to infinity. He showed that if the outer billiard curve $M$ is smooth enough and strictly convex then the map $T$ far away from $M$ is a small perturbation of an integrable map, and therefore, by the KAM theory, it possessed invariant curves that prevented the orbits from escaping; see  \cite{Mo1,Mo2} and \cite{Do}.\footnote{Alas, R. Douady's thesis is not published and is hard to obtain.}

The outer billiard map can be defined in the symplectic space of any (even) dimension, of which the plane with its standard area form is an example.  The aim of this paper is to investigate the asymptotic behavior of this map in the multi-dimensional setting. The definition of the map is as follows.

Let $M$ be a smooth closed quadratically convex hypersurface in $\R^{2d}$ with its standard symplectic structure $\omega$. At every point $m \in M$ the oriented characteristic direction is defined: it is the kernel of the restriction of $\omega$ to the tangent hyperplane $T_m M$. If one identifies $\R^{2d}$ with $\C^d$, the characteristic direction is obtained from the outward normal by multiplying it by $\sqrt{-1}$. 

Given a point $x$ outside of $M$, one can prove that there is a unique point $m \in M$ such that the oriented line $x m$ is the oriented characteristic at $m$ (there is also a unique point $m' \in M$ such that the line $m'x$ is the oriented characteristic at $m'$). The {\it outer symplectic billiard map} $T$ is the reflection of $x$ through $m$. This map preserves the symplectic structure. See \cite{Ta1,Ta3} and, most recently, \cite{ACT} for details.

A peculiar property of the outer billiard map $T$ in the plane is that, in the first approximation, the orbits of $T^2$ ``at infinity" lie on homothetic centrally symmetric closed convex curves whose shape is determined by the outer billiard curve $M$ (namely, they are polar dual to the central symmetrization of  $M$), see \cite{Ta2} and the survey \cite{DT}. 

One can observe this behavior on a computer screen by rescaling so that $M$ appears as almost a point, the outer billiard map as almost the reflection in this point, and the evolution of a point under $T^2$ as almost a continuous motion. For example, if $M$ is the Reuleaux triangle, then its central symmetrization is a circle, and the dynamics of $T^2$ appear as a uniform circular motion, see Example \ref{tri}. This phenomenon is illustrated in Figure \ref{lp}.

\begin{figure}[hbtp] 
\centering
\includegraphics[width=2.5in]{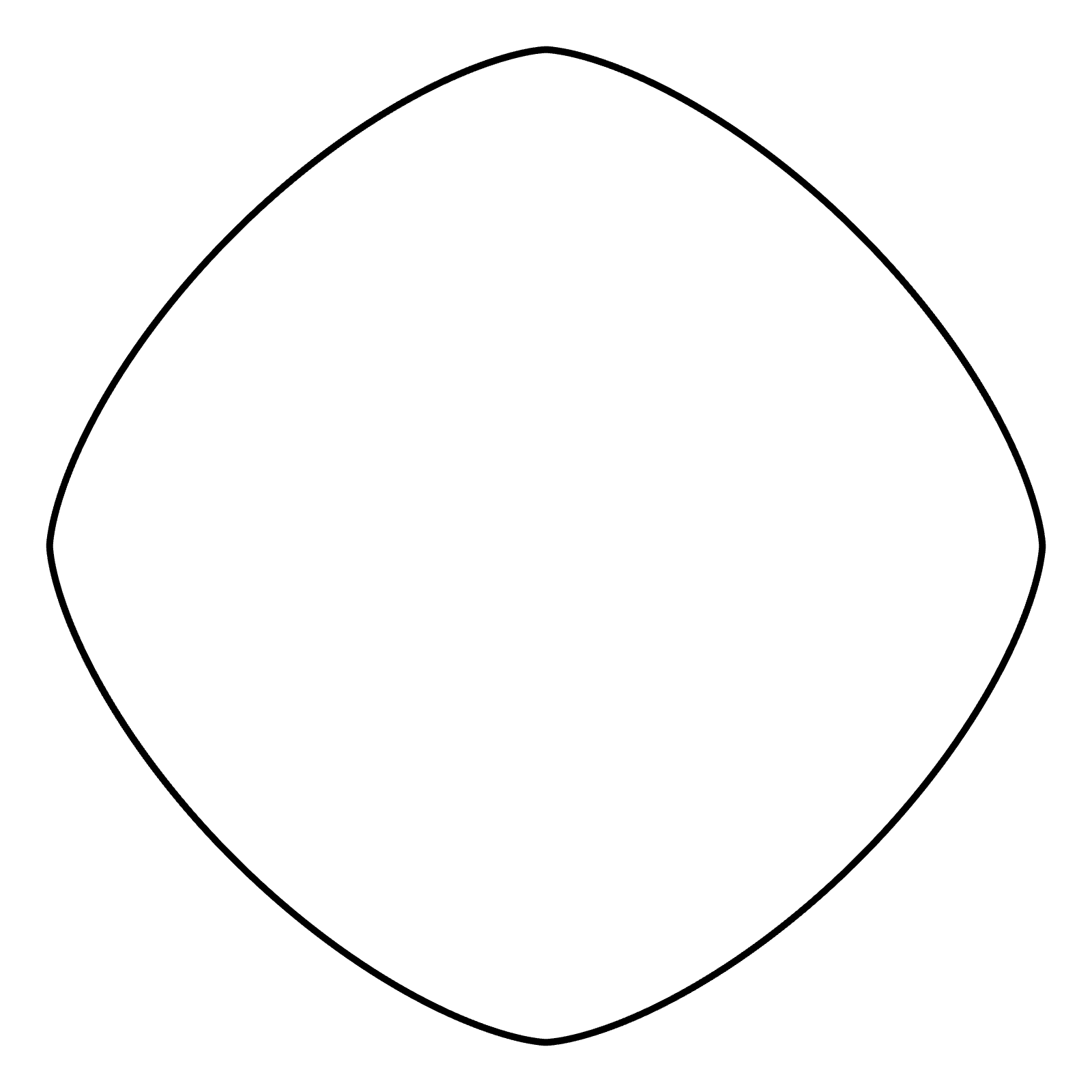}\qquad
\includegraphics[width=2.5in]{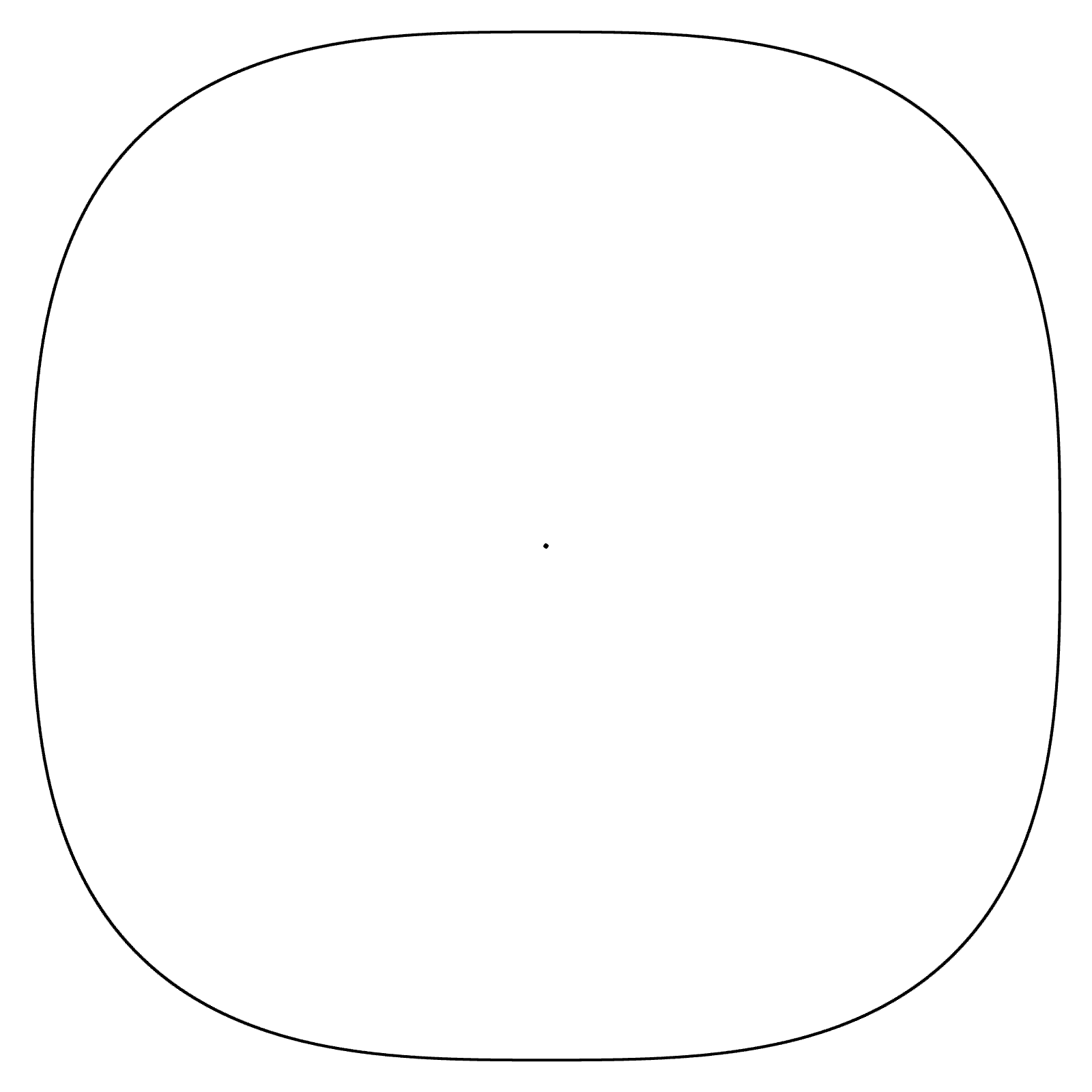}
\caption{Left: the outer billiard table is a unit circle in the $p$-norm with $p=1.5$. Right: an orbit of this outer billiard ``at infinity": it is a circle in the $p$-norm with $p=3$ (these two norms are dual to each other). The dot at the center is the outer billiard table from the left to scale.}
\label{lp}
\end{figure}
 
An informal explanation of this phenomenon is as follows. 

Given a direction $\xi$ in the plane, there is a unique support line to $M$ having this direction and a unique support line having the opposite direction.  Denote by $v(\xi)$ the vector that connects the tangency points of these support lines, see Figure \ref{supp}. 

\begin{figure}[hbtp] 
\centering
\includegraphics[width=3.7in]{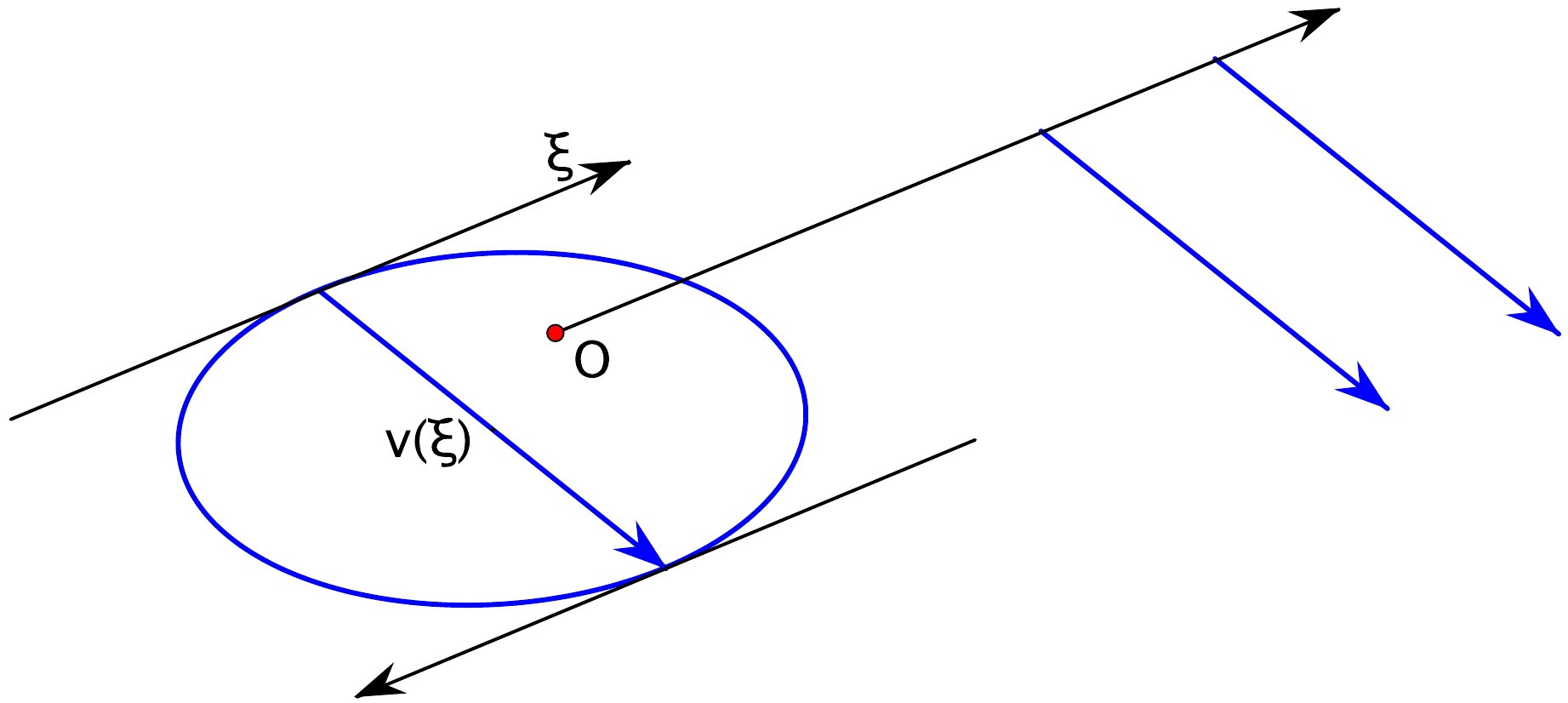}
\caption{Associating a vector with a direction and constructing a homogeneous vector field of degree zero.}
\label{supp}
\end{figure}

Assume that the origin $O$ is inside $M$. Define a homogeneous vector field of degree zero $V$ in the plane by placing the foot point of the vector $2v(\xi)$ at every point of the ray in the direction $\xi$ emanating from $O$. The integral curves of this vector field are the observed homothetic invariant curves ``at infinity". 

Indeed, consider the map $T^2$ far away from the origin, see Figure \ref{far}. The two support lines involved are nearly parallel and, by similar triangles, the vector $T^2(x)-x$ is twice the vector that connects the tangency points. Hence, $T^2(x)-x$ is approximated by the vector $V(x)$ and, after rescaling, the evolution of a point under $T^2$ appears on the computer screen as the time-1 flow of the vector field $V$. The integral curves of this vector field are homothetic copies of $(M\ominus M)^*$.

\begin{figure}[hbtp] 
\centering
\includegraphics[width=5.3in]{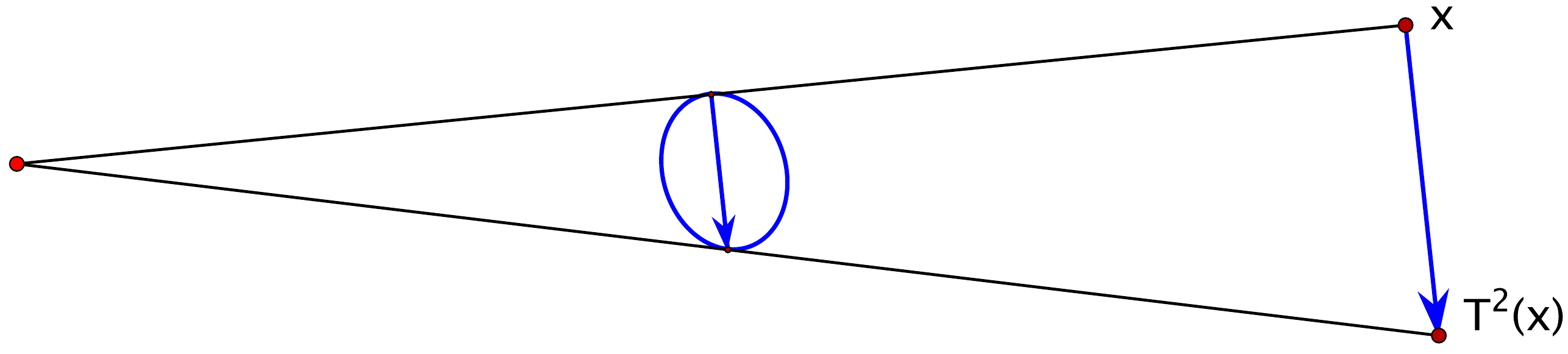}
\caption{The second iteration of the outer billiard map far away from the ``table".}
\label{far}
\end{figure}

Our main results, extending from dimension 2 to higher dimensions, comprise three theorems in Section \ref{sect:thm}. 

First, we construct a homogeneous of degree zero Hamiltonian vector field in $\R^{2d}$ whose Hamiltonian function is determined by the outer ``billiard table" $M$. We show that the second iteration of the outer symplectic billiard map is approximated by the time-$2$ flow of this field (Theorem \ref{thm:shadow}). 

Unlike the planar case, there is no general reason for orbits of the outer symplectic billiard map to stay bounded: even if the map is close to integrable, the surviving invariant tori do not separate and some orbits may escape through the gaps between them (Arnold diffusion). We show, in Theorem \ref{thm:esc}, that the rate of escape of the orbits of the outer symplectic billiard map to infinity is at most of order $\sqrt{k}$, where $k$ is the number of iterations. However, we expect the actual rate of escape to be substantially smaller. 

Not much is known about periodic points of the outer symplectic billiard map beyond the fact that, for every odd $k \ge 3$, there exists a $k$-periodic orbit (see \cite{ACT,Ta1,Ta3}). 
In Theorem \ref{thm:per}, we show that the periodic points at a large distance from the ``table" cannot have small periods. 


We finish this introduction by comparing the problem discussed in the present paper with a somewhat similar problem concerning the conventional (inner) billiards in a domain in $\R^d$ bounded by a closed strictly convex hypersurface $M$. 

An analog of the outer symplectic billiard orbits far away from the outer billiard table are glancing billiard trajectories that make small angles with $M$, that is, the orbits of the billiard ball map that are close to the boundary of the phase space. The theory of interpolating Hamiltonians, see \cite{Me,PS}, implies that the billiard ball map equals the time-1 map of a Hamiltonian vector field, composed with a smooth symplectic map that fixes the boundary of the phase space point-wise to all orders.

We also approximate the second iteration of the outer symplectic billiard map ``at infinity" by the flow of a Hamiltonian vector field whose Hamiltonian function is determined by the outer billiard table, but this approximation is considerably weaker than that in the theory of interpolating Hamiltonians.

\bigskip

{\bf Acknowledgements}. We thank Lael Costa for interesting discussions and numerical experiments.

PA acknowledges funding by the Deutsche Forschungsgemeinschaft (DFG, German Research Foundation) through Germany's Excellence Strategy EXC-2181/1 - 390900948 (the Heidelberg STRUCTURES Excellence Cluster), the Transregional Collaborative Research Center CRC/TRR 191 (281071066). 
ACC was funded as well by the DFG through Project-ID 281071066 – TRR 191.
ST was supported by NSF grant DMS-2404535, Simons Foundation grant MPS-TSM-00007747, and by a Mercator fellowship within the CRC/TRR 191. He thanks the Heidelberg University for its invariable hospitality.

\section{Symplectic polar duality and symmetrization} \label{sect:convex}

We start by discussing polar duality, in particular, in the symplectic setting. For that, we assume throughout that $M \subset W$ bounds, in the vector space $W$, a quadratically convex domain containing the origin $O$ in its interior. 
Under these convexity assumptions, for every $x\in M$, there is a unique $\xi\equiv\xi(x) \in W^*$  given by
\begin{equation} \label{eq:pol}
\ker\xi = T_x M\ \ {\rm and}\ \ \xi\cdot x =1,
\end{equation}
where $\xi\cdot x$ is the dual pairing/evaluation. This defines a map $x\mapsto \xi$ and its image is the polar dual hypersurface $M^*\subset W^*$. Due to our convexity assumptions, $M^*$ is diffeomorphic to $M$. 

Recall that polar duality is an involution: $(M^*)^*=M$.
Indeed, the first equation in (\ref{eq:pol}) means that the 1-form $\xi dx$ vanishes on $M$, and the differential of the second equation yields $\xi dx+xd\xi=0$. Hence the 1-form $xd\xi$ vanishes on $M^*$.

If $W$ is Euclidean space, then one may identify $W$ with $W^*$ using the inner product, and therefore consider $M^*$ as a hypersurface in $W$.

Consider now $(W,\omega)$, a symplectic vector space. This time, we identify $W$ with $W^*$ using the symplectic structure as follows. The vector $x\in W$ is identified with the covector $\omega(\cdot,x)\in W^*$. Therefore, as above, for every $x\in M$, there is a unique $R\equiv R(x)\in W$ with the properties
\begin{equation} \nonumber
\ker\omega(\cdot,R) = T_x M\ \ {\rm and}\ \ \omega(x,R) =1,
\end{equation}
that is,
\begin{equation*} \label{eq:spol}
\omega(v,R) = 0\;\;\forall v\in T_x M\ \ {\rm and}\ \ \omega(x,R) =1.
\end{equation*}
In other words, the vector $R(x)$ is the Reeb vector field of the contact form $\omega(x,\cdot)|_{T_xM}$. The image $M^*\subset W$ of the map $x\mapsto R(x)$ is the {\it symplectic polar} of $M$. 

\begin{remark}
{\rm The notion of the {symplectic polar} duality has recently appeared in the literature, see \cite{Be,BB,BK,GG}. By identifying $\R^{2d}=\C^d$, the symplectic polar differs from the Euclidean polar by the complex rotation $J$, given by the multiplication by $\sqrt{-1}$.
}
\end{remark}

\begin{example} \label{ex:2d}
{\rm Let $\gamma(t)$ be a star-shaped curve with respect to the origin, i.e., $\omega(\gamma,\gamma')>0$. Its symplectic polar is the curve $$
\gamma^*(t)=\frac{\gamma'(t)}{\omega(\gamma(t),\gamma'(t))}.
$$
Indeed, $\omega(\gamma(t),\gamma^*(t))=1$ and $\omega(\gamma'(t),\gamma^*(t))=0$. If, in addition, $\gamma$ is a smooth quadratically convex curve, then its symplectic polar curve $\gamma^*$ is also smooth and convex. 
}
\end{example}

Let us add an extra assumption: $M$ is the unit hypersurface of a positively homogeneous function $f: W \to \R$ of degree one.
Its Hamiltonian vector field $X_f$ is defined by $\omega(\cdot, X_f) = df (\cdot)$. Recall that the Hamiltonian vector field of any function, whose level set is $M$, points into the characteristic direction $\ker\omega|_{T_xM}=\ker\omega(x,\cdot)|_{T_xM}$ of $M$ at $x$.

\begin{lemma} \label{lm:sgr}
For all $x$ in $M=\{f=1\}$ we have the equality
\begin{equation}\nonumber
X_f(x)=R(x)\;.
\end{equation}
\end{lemma}

\begin{proof} 
Since $X_f(x)$ already points into the characteristic direction it suffices to compute
$$
\omega(x, X_f(x))= df(x)x 
=f(x)=1,
$$
where we used the Euler Homogeneous Function Theorem in the second equality.
\proofend
\end{proof}

To simplify the notation, we write $a\sim b$ for two vectors $a$ and $b$ that are proportional with a positive coefficient.

\begin{lemma}\label{lem:outer_normal_Reeb_for_exterior_point}
For a point $x$ in the exterior of $M$, there exists a unique point $n_-(x)\in M$, resp.~$n_+(x) \in M$, such that $ R(n_-(x)) \sim -x$, resp.~$ R(n_+(x)) \sim x$. Explicitly, in terms of the normal Gauss map $G:M\to S^{2d-1}$, we have
\begin{equation}\nonumber
n_+(x)=G^{-1}\left(-\frac{Jx}{|x|}\right),\quad n_-(x)=G^{-1}\left(\frac{Jx}{|x|}\right).
\end{equation}
\end{lemma}

\begin{proof}
To simplify notation, we work in $W=\R^{2d}$ with its standard inner product, symplectic form, and complex structure $J$. Due to the convexity properties of $M$, the normal Gauss map
\begin{equation}\nonumber
\begin{aligned}
G:M&\to S^{2d-1}\\
q&\mapsto \frac{\nabla f(q)}{|\nabla f(q)|}
\end{aligned}
\end{equation} 
is a diffeomorphism. For $x$ in the exterior of $M$ we consider the point
\begin{equation}\nonumber
q:=G^{-1}\left(\frac{Jx}{|x|}\right)\in M.
\end{equation}
Then $R(q)=X_f(q)=J\nabla f(q)$, together with 
\begin{equation}\nonumber
\nabla f(q)\sim\frac{\nabla f(q)}{|\nabla f(q)|}=G(q)=\frac{Jx}{|x|}\sim Jx,
\end{equation}
implies that $R(q)=J\nabla f(q)\sim -x$. In other words,
\begin{equation}\nonumber
n_-(x):= q =G^{-1}\left(\frac{Jx}{|x|}\right).
\end{equation}
Similarly, we obtain
\begin{equation}\nonumber
n_+(x):=G^{-1}\left(-\frac{Jx}{|x|}\right).
\end{equation}
\proofend
\end{proof}

We point out that forming the symplectic polar is not quite a duality. 

\begin{lemma}
If $R^*$ denotes the Reeb vector field of $M^*$, then 
$$
R^*\circ R=-\mathrm{id}_M.
$$ 
In particular, 
\begin{equation}\nonumber
(M^*)^*=-M.
\end{equation}
\end{lemma}

\begin{proof}
Linearizing the equation $\omega(x,R(x))=1$ we obtain
\begin{equation}\nonumber
\omega(x, DR(x)\hat x)=\omega(R(x),\hat x)
\end{equation}
for all $x\in M$ and $\hat x\in T_xM$.
Next, we set $a:=R(x)\in M^*$ and verify the equality
\begin{equation}\nonumber
R^*(a)=-x.
\end{equation}
We first check that $x$ points in the characteristic direction of $M^*$ in $a$, that is, $\omega(x,u)=0$ for all $u\in T_aM^*$.

For this we use $DR(x):T_xM\to T_{a}M^*$ to write $u=DR(x)\hat x$ for some $\hat x\in T_xM$ and the previous identity to compute
\begin{equation}\nonumber
\begin{aligned}
\omega(x,u)&=\omega(x,DR(x)\hat x)\\
&=\omega(R(x),\hat x)\\
&=0,
\end{aligned}
\end{equation}
by definition of $R(x)$. The equation $\omega(a, R^*(a))=1$ for $R^*(a)=-x$ follows almost tautologically from the one for $a=R(x)$:
\begin{equation}\nonumber
1=\omega(x, R(x))=\omega(x,a)=\omega(a,-x).
\end{equation}
We conclude that indeed $R^*(a)=-x$ for $a=R(x)$, i.e., $R^*(R(x))=-x$.
\proofend 
\end{proof}


This lemma leads us to the notion of central symmetrization. Before that, we prove the following.

\begin{lemma}\label{lem:duality_for_n_plus}
We write, as above, $M=\{f=1\}$ and similarly $M^*=\{f^*=1\}$ for 1-homogeneous functions $f,f^*:\R^{2d}\to\R$. For a point $x$ in the exterior of $M$ and $M^*$ we have the equalities
\begin{equation}\nonumber
-n_+(x)=R^*\left(\frac{x}{f^*(x)}\right), \quad -n_-(x)=R^*\left(\frac{-x}{f^*(-x)}\right).
\end{equation}
\end{lemma}

\begin{proof}
Recall from above that $R(n_+(x))\sim x$ and $R(n_+(x))\in M^*$ to conclude
\begin{equation}\nonumber
R(n_+(x))=\frac{x}{f^*(x)}.
\end{equation}
Using $R^*\circ R=-\mathrm{id}_M$ we see directly the first claimed equality. The second follows similarly.
\proofend
\end{proof}

Next, we recall the process of symmetrization for a smooth closed quadratically convex hypersurface $M\subset \R^d$, and assume that the origin $O$ is inside $M$. The normal Gauss map $G:M \to S^{d-1}$ is a diffeomorphism, and one can use its inverse to parameterize  $M$. This is done via the {\it support function} $p:S^{d-1}\to\R$, the distance from $O$ to the tangent hyperplane of $M$ as a function of the unit outer normal vector. Namely, if $v \in S^{d-1}$ is an outward normal vector at the point $x\in M$, then one has
$$
x(v)=p(v)v+\nabla p(v),
$$
where the gradient is taken in the standard spherical metric (see, e.g., \cite{San}). In particular, we have the equality $G(x(v))=v$, i.e., $x(v)=G^{-1}(v)$.

The {\it symmetrization} of $M$ is the centrally symmetric hypersurface $\overline M$ whose support function is  $\bar p(v)=p(v)+p(-v)$. The origin $O$ is the center of $\overline M$. Equivalently, $\overline M$ is the Minkowski sum of $M$ and $-M$, its centrally symmetric image with respect to point $O$.

\begin{figure}[h] 
\centering
\includegraphics[width=2.5in]{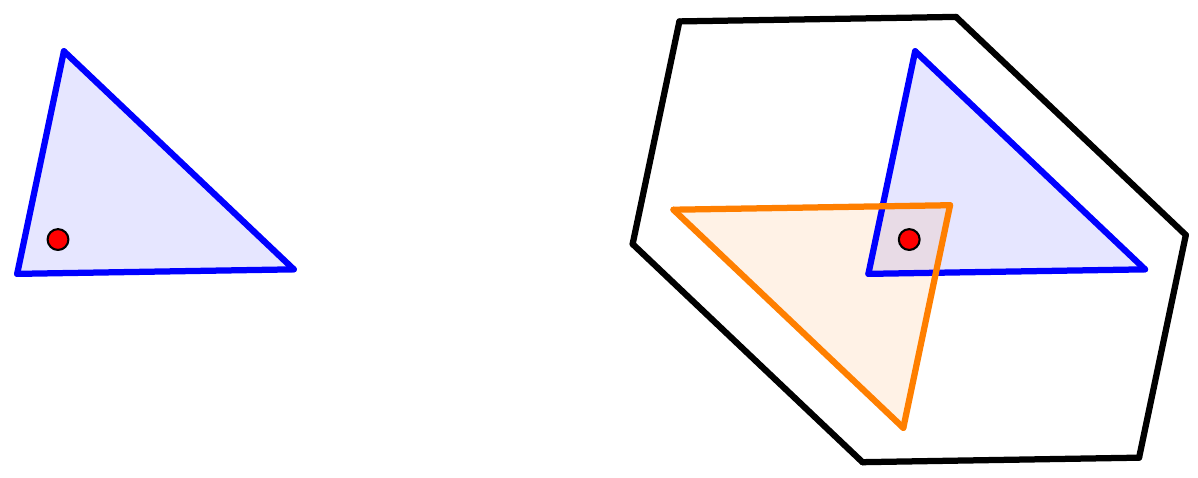}
\caption{A triangle and its symmetrization.}
\label{triangles}
\end{figure}

\begin{example} \label{tri}
{\rm  
We provide some examples of symmetrization. A trivial example is that of a centrally symmetric body: its symmetrization is its homothetic copy, twice as large. 
\begin{enumerate}   \renewcommand{\theenumi}{\roman{enumi}}\renewcommand{\labelenumi}{(\theenumi)} 
\item Let $M$ be a triangle. Its symmetrization is an affine-regular hexagon, see Figure \ref{triangles}.

\item Let $M$ be a figure of constant width. Since $p(v)+p(-v)$ is constant, $\overline M$ is a circle, see Figure \ref{Reuleaux}.

\item Let $M$ be a semicircle. Its symmetrization is a stadium, the convex hull of two tangent congruent circles. The symplectic polar of the stadium is bounded by two arcs of parabolas, sharing the axis and intersecting at right angles, see Figure \ref{halfcircle}. 
\end{enumerate}
}
\end{example}

\begin{figure}[h] 
\centering
\includegraphics[width=3in]{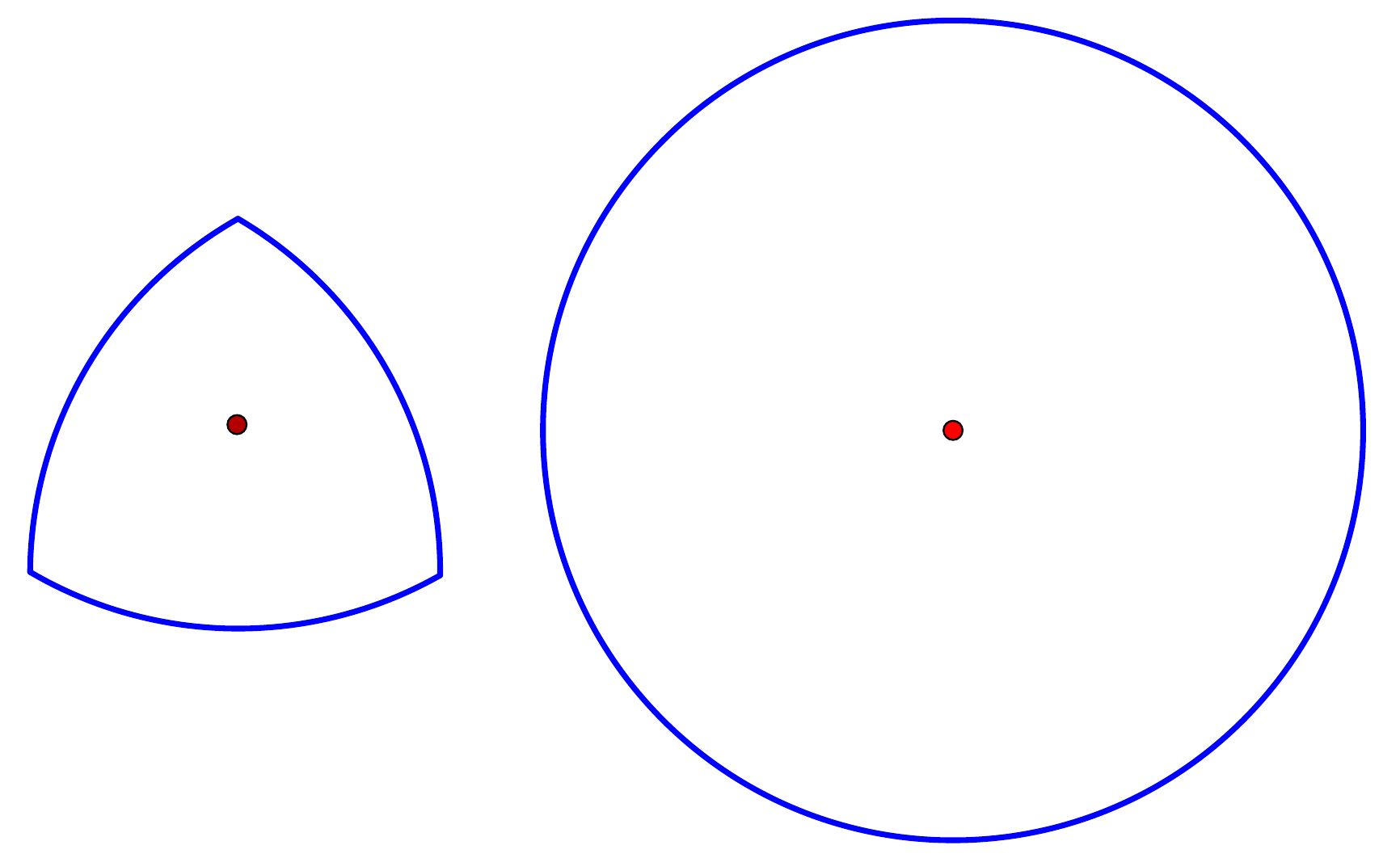}
\caption{A Reuleaux triangle and its symmetrization.}
\label{Reuleaux}
\end{figure}

\begin{figure}[h] 
\centering
\includegraphics[width=4 in]{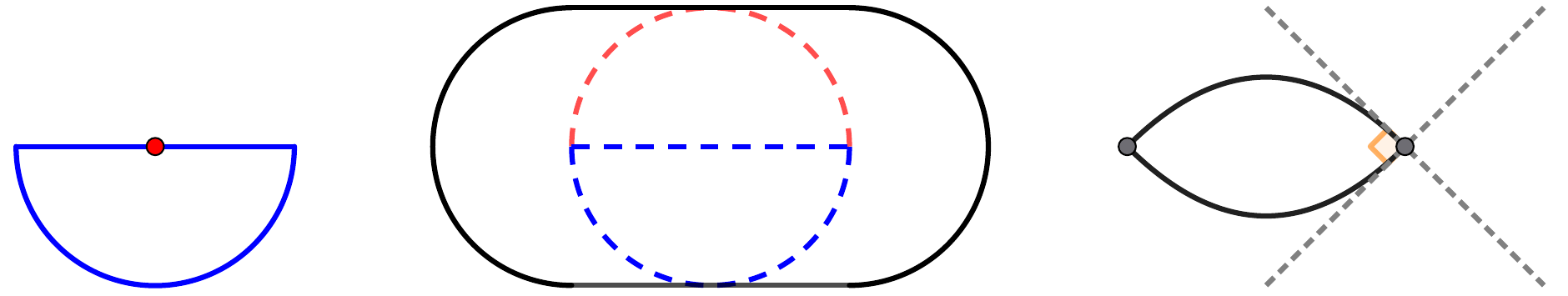}
\caption{The symmetrization of a half-circle and the symplectic polar of the symmetrization.}
\label{halfcircle}
\end{figure}

The following Lemma computes the inverse of the Gauss map of the symmetrization.

\begin{lemma} \label{lm:symm}
Let $\bar x(v) \in \overline M$ be the point whose unit outer normal is $v$. Then $\bar x(v)=x(v)-x(-v)$. That is, the Gauss map $\overline G:\overline M\to S^{d-1}$ of the symmetrization $\overline M$ of $M$ satisfies
\begin{equation}\nonumber
\overline G^{-1}(v)= G^{-1}(v)- G^{-1}(-v).
\end{equation}
\end{lemma}

\begin{proof}
Recall if $p$ is the support function of $M$ then $\overline M$ has, by definition, support function $\bar p(v):=p(v)+p(-v)$. Therefore,
\begin{equation}\nonumber
\begin{aligned}
\bar x(v)&=\bar p(v)v+\nabla \bar p(v)\\ 
&= p(v)v+\nabla p(v) - p(-v)v-\nabla p(-v)\\
&= x(v)-x(-v),
\end{aligned}
\end{equation}
as claimed.
\proofend
\end{proof}

Now, we go back to the situation when $M\subset (\R^{2d},\omega)$ and denote all quantities associated to $\overline M$ with a bar, e.g., $\overline R$, $\overline n_\pm$, etc.

\begin{lemma}\label{lem:bar_R_and_bar_n}
Let $x\in\R^{2d}$ be in the exterior of $\overline M$. Then 
\begin{equation}\nonumber
\overline n_+(x)=n_+(x)-n_-(x) =-\overline n_-(x)
\end{equation}
and, in particular, we have the relations
\begin{equation}\nonumber
\overline R (n_+(x)-n_-(x))\sim x,\quad \overline R (n_-(x)-n_+(x))\sim -x.
\end{equation}
\end{lemma}

\begin{proof}
Combining Lemma \ref{lem:outer_normal_Reeb_for_exterior_point} applied to $M$ and $\overline M$ with Lemma \ref{lm:symm} we see that
\begin{equation}\nonumber
\begin{aligned}
\overline n_+(x)&=\overline G^{-1}\left(-\frac{Jx}{|x|}\right)\\
&=G^{-1}\left(-\frac{Jx}{|x|}\right)-G^{-1}\left(\frac{Jx}{|x|}\right)=-\overline n_-(x)\\
&=n_+(x)-n_-(x).
\end{aligned}
\end{equation}
The definition $\overline R (\overline n_\pm(x))\sim \pm x$ implies then the remaining relations.
\proofend 
\end{proof}

\section{The second iteration of the outer billiard map near infinity} \label{sect:thm}

Now we study the behavior of the outer symplectic billiard map for points far away from the ``outer billiard table" $M$. 

From now on we assume, for notational simplicity, that $M \subset (\R^{2d},\omega)$ bounds a strictly convex domain (the ``interior'') containing the origin $O$ in its interior and, moreover, that $M$ can be written as a level set of a smooth function with a positive definite Hessian. Denote by $ R(m)$ the Reeb vector field at the point $m\in M$. 

Due to our convexity assumption, for any point $x$ in the exterior of $M$, there exists a unique point $m_-(x) \in M$ such $m_-(x) -x\sim  R(m_-(x))$, see, e.g., \cite{Ta1} and see Figure \ref{vectors2}. In other words, there is a unique ray emanating from $x$ touching $M$ in $m_-(x)$ and pointing in the \emph{positive} characteristic direction. 

We briefly recall the proof from \cite{Ta1}. First, we observe that the map $M\times(0,\infty)\to\R^{2d}$, $(q,t)\mapsto q+tR(q)$, continuously extends to $M\times\{0\}\to M$ as the identity and to $M\times\{\infty\}\to S^{2d-1}_\infty$, the sphere at infinity / of directions, as $q\mapsto \R_{>0}R(q)$. Due to convexity, the map at infinity has mapping degree 1. If we extend this map over the interior of $M$ by the identity, we obtain a continuous map of the closed $2d$-ball which restricts to a degree-1 map on the boundary. Such a map is necessarily onto, e.g., using homological arguments.

Surjectivity means that through each point $x$ in the exterior of $M$ passes at least one positive characteristic direction ray of $M$. Uniqueness follows by arguing by contradiction. Assume that we find $q_1,q_2\in M$, $t_1,t_2>0$, such that $x=q_1+t_1R(q_1)=q_2+t_2R(q_2)$. Convexity implies that $\omega(q_2-q_1,R(q_2))=\langle q_2-q_1,n_{2}\rangle>0$ and $\omega(q_1-q_2,R(q_1))=\langle q_1-q_2,n_{1}\rangle>0$, see Figure \ref{genial}.
\begin{figure}[hbtp] 
\centering
\includegraphics[width=3.3in]{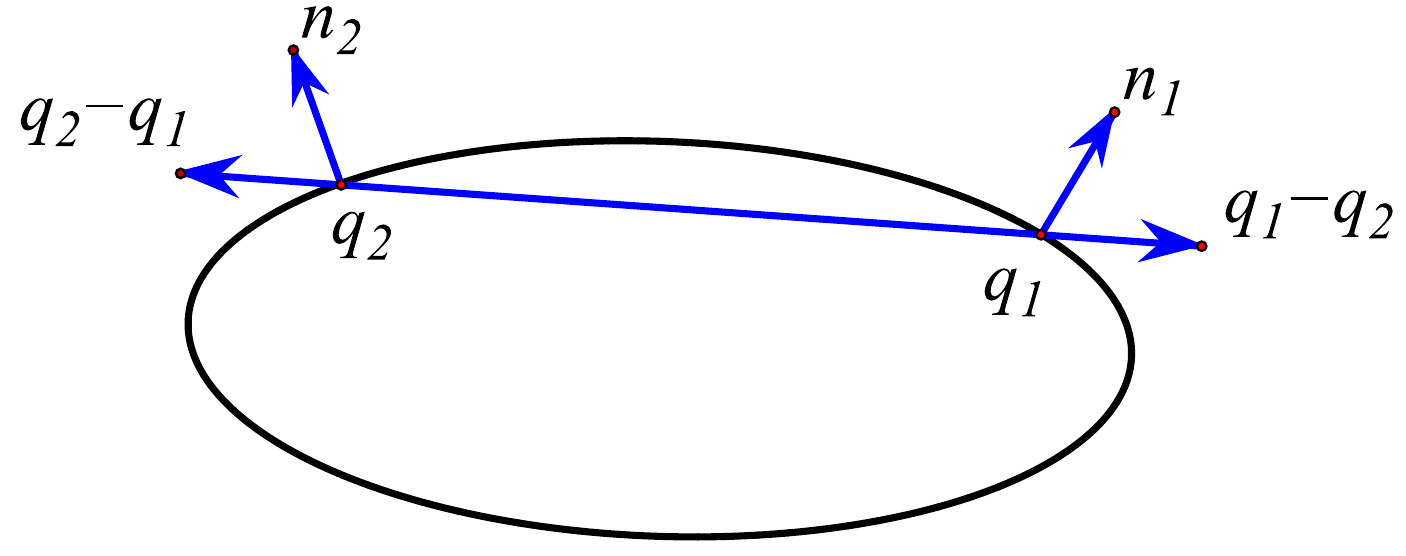}
\caption{Proving that the multiplicity of the covering by the positive characteristic rays equals 1.}
\label{genial}
\end{figure}

Combining these we arrive at
\begin{equation}\nonumber
\begin{aligned}
0&<\omega(q_1-q_2,R(q_1))=t_2\omega(R(q_2),R(q_1)),\\
0&<\omega(q_2-q_1,R(q_2))=t_1\omega(R(q_1),R(q_2)).
\end{aligned}
\end{equation}
In particular, $t_1$ and $t_2$ need to have opposite signs. This contradiction proves uniqueness.

In this new notation, the outer symplectic billiard map is given by $T(x)=2m_-(x) - x$, the reflection of $x$ in the point $m_-(x)\in M$, see Figure \ref{vectors2}.

\begin{figure}[hbtp] 
\centering
\includegraphics[width=5.3in]{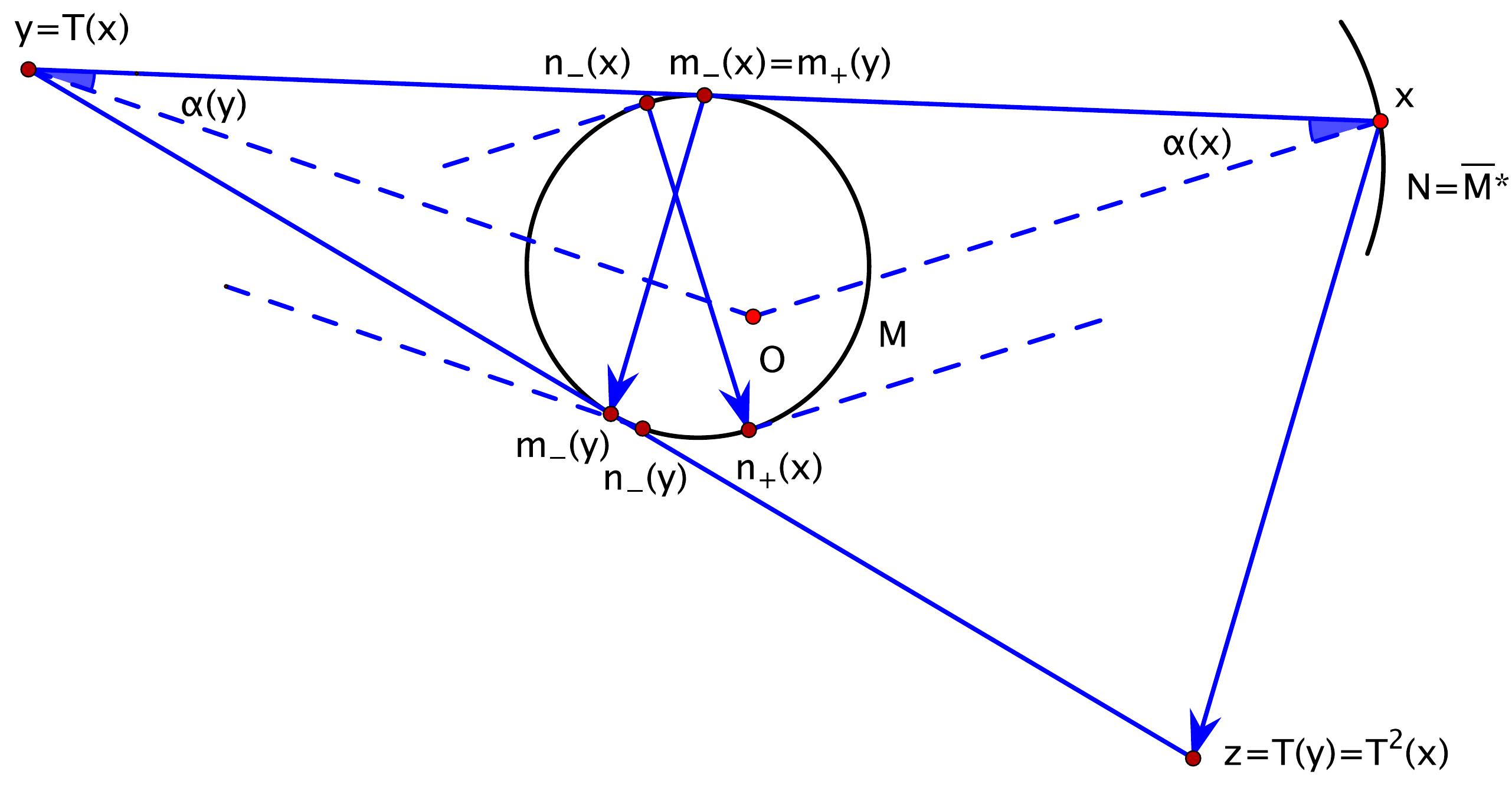}
\caption{The second iteration of the map $T$.}
\label{vectors2}
\end{figure}

There is also a unique point $m_+(x) \in M$ such that  $x-m_+(x)\sim  R(m_+(x))$, corresponding to the unique ray emanating from $x$ which is touching $M$ at $m_+(x)$ and pointing in the  \emph{negative} characteristic direction.

We already discussed above that for any point $x$ in the exterior of $M$, there exist a unique point $n_-(x)\in M$, resp.~$n_+(x) \in M$, such that $ R(n_-(x)) \sim -x$, resp.~$ R(n_+(x)) \sim x$, see again Figure \ref{vectors2}.

Recalling that $T(x)=2m_-(x) - x$ and noticing that $m_-(x)=m_+(y)$ for $y=T(x)$ we obtain
\begin{equation}\nonumber
T^2(x)-x = 2(m_-(y)-m_-(x)),
\end{equation}
or, more symmetrically,
\begin{equation}\nonumber
T(y)-T^{-1}(y) = 2(m_-(y)-m_+(y)).
\end{equation}
Now, let $\overline M=M\ominus M$ be the central symmetrization of $M$ and $N:=(\overline M)^*$ be its symplectic polar dual. 

\begin{proposition} \label{prop:Ham}
Let $H$ be the homogeneous function of degree one whose unit level hypersurface is $N$. Then for any $x$ in the exterior of $\overline M$,
$$
V(x) :=2(n_+(x)-n_-(x))= -2 X_H (x).
$$
In particular, $V$ is homogeneous of degree zero.
\end{proposition}

\begin{proof}
Recall that above we associated with the hypersurface $M$ the quantities $R$ and $n_\pm$ and with $\overline M$ the quantities $\overline R$ and $\overline n_\pm$. Note that a point $x$ in the exterior of $\overline M$ is also in the exterior of $M$. Moreover, Lemma \ref{lem:bar_R_and_bar_n} asserts 
\begin{equation}\nonumber
\overline n_+(x)=n_+(x)-n_-(x) \in \overline M
\end{equation} 
and 
\begin{equation}\nonumber
\overline R (n_+(x)-n_-(x))\sim x.
\end{equation}
See also Figure \ref{dual} for a geometric interpretation.
\begin{figure}[hbtp] 
\centering
\includegraphics[width=2.5in]{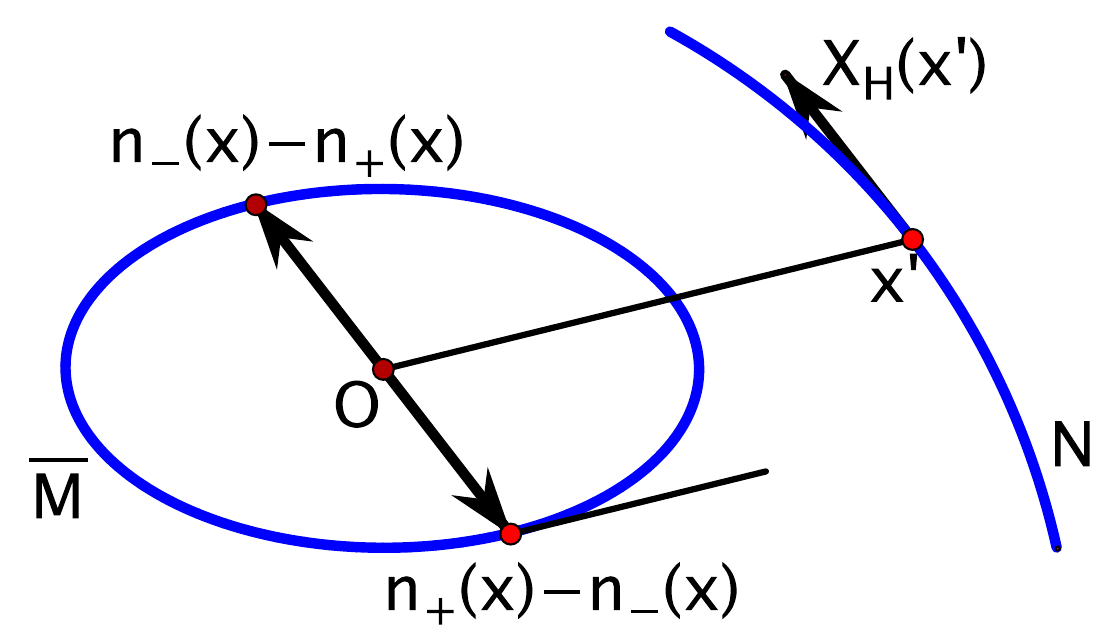}
\caption{Illustrating Proposition \ref{prop:Ham}.}
\label{dual}
\end{figure}
Now we apply Lemma \ref{lem:duality_for_n_plus} to the point $x':=\frac{x}{H(x)} \in N=\{H=1\}=(\overline M)^*$ and combine it with Lemma \ref{lm:sgr} to conclude that
\begin{equation}\nonumber
-\overline n_+(x)=X_H(x').
\end{equation}
Finally, since $H$ is homogeneous of degree one, its Hamiltonian field $X_H$ is homogeneous of degree zero, hence 
\begin{equation}\nonumber
X_H(x)=X_H(x').
\end{equation}
Putting everything together, we see that
\begin{equation}\nonumber
\begin{aligned}
X_H(x)=X_H(x')=-\overline n_+(x)=n_-(x)-n_+(x),
\end{aligned}
\end{equation}
as claimed.
\proofend
\end{proof}

Next we estimate  how close the vectors $T^2(x)-x=2(m_-(y)-m_-(x))$ and $V(x)=2(n_+(x)-n_-(x))$ are depending on the distance $|x|$ of $x$ from the origin. 

\begin{lemma} \label{lm:eps}
There exist constants $C,\delta>0$, depending only on $M$, such that for $|x| \ge \frac{1}{\delta}$ one has
\begin{equation}\nonumber
\begin{aligned}
(1)&& \quad |n_-(x)-m_-(x)| &\leq \frac{C}{|x|},\\[1ex]
(2)&& |m_-(y)-n_-(y)| &\leq \frac{C}{|x|},\\[1ex]
(3)&& |n_-(y)-n_+(x)| &\leq \frac{C}{|x|}.
\end{aligned}
\end{equation}
Here, as above, $y=T(x)$.
\end{lemma}

\begin{proof} We continue to refer to Figure \ref{vectors2} and the notation from above. We start with the following estimate for a point $x$ in the exterior of $M$. Recall that by definition of the reflection rule, we have 
\begin{equation}\nonumber
|y-m_-(x)|=|x-m_-(x)|
\end{equation}
for $y=T(x)$. Combining this with the triangle inequality, we see that
\begin{equation}\nonumber
\begin{aligned}
|x|&\leq |x-m_-(x)|+|m_-(x)|\\
&=|y-m_-(x)|+|m_-(x)|\\
&\leq |y|+2|m_-(x)|.
\end{aligned}
\end{equation}
If we set $C_1=\mathrm{diam}(M)$, then we arrive at 
\begin{equation}\label{eq:estxy_one_side}
|y|\geq |x|-2C_1.
\end{equation}
For later purposes we note that, by symmetry, the inequality $|x|\geq |y|-2C_1$ holds as well and, in combination,
\begin{equation*}\label{eq:estxy}
||x|-|y||\leq 2C_1.
\end{equation*}
Next, we consider the angle between the unit normals to $M=\{f=1\}$ at the points $n_-(x)$ and $m_-(x)$ and denote it by $\alpha(x)$, i.e.,
\begin{equation}\nonumber
\cos \alpha(x)=\left\langle\frac{\nabla f(n_-(x))}{|\nabla f(n_-(x))|},\frac{\nabla f(m_-(x))}{|\nabla f(m_-(x))|}\right\rangle,
\end{equation}
and likewise for the angle $\alpha(y)$ between the normals at the points $n_-(y)$ and $m_-(y)$. 

Recall that $J\nabla f(n_-(x))\sim -x$ and $J\nabla f(m_-(x))\sim m_-(x)-x$. Since $J$ is an isometry, we conclude that $\alpha(x)$ is also the angle between $x$ and $x-m_-(x)$, that is, 
$$
\alpha(x)=\angle m_-(x)xO.
$$
Now, we choose $\delta$ such that $|x|\geq \frac{1}{\delta}$ implies that 
$$
2\sin\alpha\geq\alpha
$$ 
for both angles $\alpha=\alpha(x)$ and $\alpha=\alpha(y)$. The inequality above holds for $\alpha \in [0, \pi/2]$, and is therefore clearly satisfied for $x$ far enough (see Figure \ref{vectors2}), and thus also for $y=T(x)$, due to  inequality \eqref{eq:estxy_one_side}.

Setting $\beta(x) = \angle O m_-(x) x$, the sine rule in the triangle $\triangle Oxm_-(x)$ implies 
$$
\sin \alpha(x) = \frac{\ |m_-(x)|}{|x|}\sin \beta(x).
$$
Recalling that $C_1=\mathrm{diam}(M)$, we can estimate
$$
\sin \alpha(x) \leq \frac{C_1}{|x|}.
$$
Now, since $|x|\geq \frac{1}{\delta}$, we have $2\sin\alpha(x)\geq\alpha(x)$, and thus
\begin{equation} \label{eq:estal}
\alpha(x) \leq  \frac{2C_1}{|x|}.
\end{equation}
In terms of the Gauss map $G:M\to S^{2d-1}$, inequality \eqref{eq:estal} asserts that the points $G(n_-(x))\in S^{2d-1}$ and $G(m_-(x))\in S^{2d-1}$ have (spherical) distance at most $\frac{2C_1}{|x|}$. Since $G$ is a diffeomorphism of compact manifolds, it is bi-Lipschitz. In particular, there exists a constant $\ell>0$ with 
\begin{equation}\nonumber
\ell |q-p|\leq \angle G(q)G(p).
\end{equation}
We conclude that
\begin{equation}\nonumber
\ell |n_-(x)-m_-(x)|\leq \angle G(n_-(x))G(m_-(x))=\alpha(x) \leq  \frac{2C_1}{|x|}
\end{equation}
that is
\begin{equation}\nonumber
|n_-(x)-m_-(x)|\leq  \frac{2C_1}{\ell|x|}.
\end{equation}
Therefore, we proved Claim (1) with the constant $\frac{2C_1}{\ell}$ and, in fact, also Claim (2), since we can repeat the argument for $\alpha(y)$ as it also satisfies $2\sin\alpha(y)\geq\alpha(y)$. Using inequality (\ref{eq:estxy_one_side}) and assuming $\frac{1}{\delta} \geq 2C_1$, we obtain: 
\begin{equation}\nonumber
|n_-(y)-m_-(y)|\leq \frac{2C_1}{\ell|y|}\leq \frac{4C_1}{\ell|x|}.
\end{equation}
Next, let us denote by $\gamma$ the angle between the normals to $M$ at the points $n_-(y)$ and $n_+(x)$, i.e.,
\begin{equation}\nonumber
\cos\gamma=\left\langle\frac{\nabla f(n_-(y))}{|\nabla f(n_-(y))|},\frac{\nabla f(n_+(x))}{|\nabla f(n_+(x))|}\right\rangle.
\end{equation}
Again, $J\nabla f(n_+(x))\sim x$ and $J\nabla f(n_-(y))\sim -y$ imply
\begin{equation}\nonumber
\begin{aligned}
\gamma &= \pi - \angle xOy\\
&= \angle Oxy + \angle Oyx\\
&=\alpha(x)+\alpha(y)\\
&\leq \frac{6C_1}{|x|},
\end{aligned}
\end{equation}
where we applied \eqref{eq:estal} to $\alpha(x)$ and $\alpha(y)$. This proves Claim (3) with constant $6C_1/\ell$. Setting
\begin{equation}\nonumber
C:=\frac{6C_1}{\ell}
\end{equation} 
proves the Lemma. 
\proofend
\end{proof}

To prove Theorem \ref{thm:shadow} below, we need a uniform estimate for the vector field $V(x)=2(n_+(x)-n_-(x))$ and its flow, which we denote by $\varphi_t$.

\begin{lemma}\label{lm:taylor}
There exist constants $\tilde C, \tilde\delta> 0$ such that for all $x$ with $|x|\geq 1/\tilde\delta$ the estimate
\begin{equation}\nonumber
|\varphi_1(x)-x-V(x)|\leq \frac{\tilde C}{|x|}
\end{equation}
holds.
\end{lemma}

\begin{proof}
We first claim that there exist constants $\tilde C\geq 0$ and $\tilde\delta>0$ such that for all $q$ with $|q|\leq 1$ and $|t|\leq\tilde\delta$ we have 
\begin{equation}\nonumber
\varphi_t(q)=q+V(q)t+r(t)\quad\text{and}\quad |r(t)|\leq\tilde Ct^2.
\end{equation}
This is simply the Taylor expansion of the smooth map $t\mapsto\varphi_t(q)$, uniformly on the compact 1-ball. 

Now, let $S_c(x)=cx$ be the homothety with coefficient $c$. Since the vector field $V$ is homogeneous of degree zero, i.e., $V(x)=V(cx)$ (see Proposition \ref{prop:Ham}), its flow $\varphi_t$ satisfies
$$
\varphi_{ct} \circ S_c = S_c \circ \varphi_t \quad\text{and, in particular,}\quad \varphi_1(x)=\frac1c\varphi_c(cx).
$$
For $x$ with $|x|\geq 1/\tilde\delta$ we set $c:=\frac{1}{|x|}$ and $q:=cx$. Then we may apply the above Taylor expansion with $t=c$ to conclude
\begin{equation}\nonumber
\begin{aligned}
\left|\varphi_1(x)-x-V(x)\right|&=\left|\frac1c\big(\varphi_{c}(cx)-cx\big)-V(cx)\right|\\
&=\left|\frac1c\big(\varphi_{c}(q)-q\big)-V(q)\right|\\
&=\left|\frac{r(c)}{c}\right|\leq \tilde Cc=\frac{\tilde C}{|x|},
\end{aligned}
\end{equation}
as claimed.
\proofend
\end{proof}

Now, we are in the position to prove our first main theorem that compares the second iteration $T^2$ of the outer billiard map with the time-1 map of the vector field $V$. Set $F=\varphi_1$.

%
%

\begin{theorem} \label{thm:shadow}
Recall the constants $C, \tilde C, \delta, \tilde\delta$ from Lemmas \ref{lm:eps} and \ref{lm:taylor}. Then, for all $|x| \ge \frac{1}{\Delta}:=\max\{\frac{1}{\delta},\frac{1}{\tilde\delta}\}$, one has
\begin{equation*} \label{eq:est1}
|T^2(x) - F(x)| \le \frac{6C+\tilde C}{|x|}.
\end{equation*}
\end{theorem} 

\begin{proof}
We recall that 
$$
V(x)=2(n_+(x)-n_-(x)),\ \  T^2(x)-x = 2(m_-(y)-m_-(x))
$$
and assume that $|x| \ge \frac{1}{\Delta}=\max\{\frac{1}{\delta},\frac{1}{\tilde\delta}\}$. Since $|x|\geq\frac{1}{\delta}$ we may apply Lemma \ref{lm:eps} and, together with the triangle inequality, it follows that
$$
|(m_-(y)-m_-(x)) - (n_+(x)-n_-(x))| \leq\frac{3C}{|x|}.
$$
In particular,
\begin{equation} \label{eq:TV}
|T^2(x)-x-V(x)| \leq\frac{6C}{|x|}
\end{equation}
holds. Moreover, Lemma \ref{lm:taylor} asserts for $|x|\geq \frac{1}{\tilde\delta}$ that 
\begin{equation}\nonumber
|F(x)-x-V(x)|\leq \frac{\tilde C}{|x|},
\end{equation}
where we recall the abbreviation $F(x)=\varphi_1(x)$. The triangle inequality implies that
\begin{equation}\nonumber
|T^2(x) - F(x)|\leq \frac{(6C+\tilde C)}{|x|}
\end{equation}
for all $|x|\geq \frac{1}{\Delta}$.
\proofend
\end{proof}

Next, we estimate the rate at which the orbits of the outer symplectic billiard map may escape to infinity. For that, recall that 
$$
N=(\overline M)^*=\{H=1\}
$$
and that $H$ is 1-homogeneous. In particular,  the Hamiltonian function $H$ defines a norm in $\R^{2d}$. Thus, $H$ is equivalent to the Euclidean norm, that is, there exists a constant $\mu>0$ such that, for all $x$,
 $$
\frac1\mu |x| \leq H(x) \leq \mu |x|.
$$
This makes it possible to replace the estimates involving $|x|$ with those involving $H(x)$ by changing the relevant constants. For notational convenience, let us write
\begin{equation}\nonumber
|x|_H:=H(x).
\end{equation}
Informally speaking, the next theorem states that the escape rate of an orbit to infinity is at most of order $\sqrt{k}$ where $k$ is the number of iterations. In the proof, we need the elementary Lemma \ref{micro_lemma}, which we prove at the end of this section.

\begin{theorem} \label{thm:esc}
There is a constant $\bar C$ such that for $|x|_H \ge \frac{\mu}{\Delta}$ and all $k$ one has
$$
\big||T^{2k}(x)|_H^2 - |x|_H^2\big| \leq \bar C k.
$$
Here $\Delta$ is the constant from Theorem \ref{thm:shadow} and $\bar C$ is given explicitly by
\begin{equation}\nonumber
\bar C=2\mu^2(6C+\tilde C)+{3\mu^2(6C+\tilde C)^2\Delta^2}.
\end{equation}
\end{theorem}

\begin{proof} 
We rewrite the assertion of Theorem \ref{thm:shadow} in terms of $|\cdot|_H$, that is,  
\begin{equation*}\label{eqn:thm1_in_H_norm}
|F(x)|_H\geq\frac{\mu}{\Delta} \quad\Longrightarrow\quad |T^2(x) - F(x)|_H \le \frac{C_{2}}{|F(x)|_H}
\end{equation*}
where $C_{2}=\mu^2(6C+\tilde C)$ and where we used $H(F(x))=H(x)$, i.e., $|F(x)|_H=|x|_H$. In particular, we may apply Lemma \ref{micro_lemma} with $a=T^2(x)$ and $b=F(x)$ and constants $C=C_{2}$ and $\Lambda=\frac{\mu}{\Delta}$ to obtain
\begin{equation}\label{eq:HHH}
\begin{aligned}
||T^2(x)|_H^2-|x|_H^2|=||T^2(x)|_H^2-|F(x)|_H^2|\leq 2C_{2}+\frac{3C_{2}^2\Delta^2}{\mu^2}=: \bar C.
\end{aligned}
\end{equation}
Now we consider an orbit segment $x, T^2(x), T^4(x),\ldots,T^{2k}(x)$ that lies outside of the $|\cdot|_H$-ball of radius $\frac{\mu}{\Delta}$. Then we may apply inequality \eqref{eq:HHH} to each $||T^{2j}(x)|_H^2-|T^{2j-2}(x)|_H^2|$, $j=1,\ldots,k$. Summing these $k$ inequalities, we obtain 
$$
||T^{2k}(x)|_H^2-|x|_H^2|< \bar C k,
$$ 
as claimed.
%
%
%
%
%
%
%
%
\proofend
\end{proof}

\begin{remark}
{\rm As we mentioned in the introduction, precious little is known about the outer symplectic billiard dynamics in higher dimensions. In dimension 2, one has an example of an open set of orbits escaping to infinity with the rate of order $\sqrt{k}$. This happens when the outer billiard table is a semicircle featured in Figure \ref{halfcircle} (which is neither smooth nor strictly convex), see \cite{DF}. 

A more extreme example is that of a segment: then the map $T^2$ is a parallel translation (in the direction of the segment and distance twice its length). In this case, the rate of escape of the orbits to infinity is linear in $k$. But if the segment is replaced by a very thin ellipse, then all orbits will be confined to homothetic ellipses (the outer symplectic billiard map about an ellipse is integrable: its invariant curves are homothetic ellipses, see \cite{ACT}). 
}
\end{remark}

The next theorem concerns the periods of periodic orbits far away from the origin. Loosely speaking, such an orbit must have a large period (see Theorem \ref {thm:per} below for the exact statement). For results on the existence of periodic points, we refer the reader to \cite{ACT,Ta1,Ta3}. Note also that the outer symplectic billiard map does not have fixed points or 2-periodic points.  

Consider the compact spherical shell $S=\big\{x\ |\ \frac{1}{2} \le |x| \le  \frac{3}{2}\big\}$. Let $m >0$ be the minimum of $|V(x)|$ in $S$ for our non-vanishing homogeneous of degree 0 vector field $V$. We point out that due to 0-homogeneity $m$ is equivalently the minimum of $|V(x)|$ on $\R^{2d}\setminus\{0\}$. By uniform continuity of $V$ in $S$, we have 
\begin{equation*} 
\exists \eta > 0\ {\rm such\  that\ if}\ x,y\in S\ {\rm satisfy}\ |x-y|<\eta,\ {\rm then}\ |V(x)-V(y)| < \frac{m}{2}. 
\end{equation*}
Without loss of generality, we may assume that $\eta < \frac{1}{2}$.


\begin{lemma} \label{lm:2pts}
Let $z$ and $p$ be two points in $\R^{2d}\setminus\{0\}$ with $|z-p| < \eta |z|$.
Then one has $\frac{p}{|z|} \in S$ and $|V(z)-V(p)| < \frac{m}{2}$.
\end{lemma} 

\proof
First, we claim that
$$
\frac{1}{2} < \frac{|p|}{|z|} <\frac{3}{2},
$$
i.e., $\frac{p}{|z|} \in S$. Indeed, by the triangle inequality,
$$
|p| \le |z|+|p-z| < (1+\eta) |z| < \frac{3}{2} |z|
$$
and
$$
|p| \ge |z| - |p-z| > (1-\eta) |z| > \frac{1}{2} |z|.
$$
In total, we have 
$$
\frac{z}{|z|} \in S,\  \frac{p}{|z|} \in S,\ \left|\frac{z}{|z|} - \frac{p}{|z|}\right| < \eta.
$$
Hence, since the field $V$ is homogeneous of degree zero, 
$$
|V(z)-V(p)| = \left| V\left(\frac{z}{|z|}\right) - V\left(\frac{p}{|z|}\right)\right| < \frac{m}{2},
$$
as needed.
\proofend

We are ready for the theorem.

\begin{theorem} \label{thm:per}
Given $k$, there exists an explicit constant $\rho\equiv\rho(k,M) >0$ such that there are no $k$-periodic orbits of the outer symplectic billiard map outside of the ball of radius $\rho$.
\end{theorem} 

\begin{remark}
{\rm
The explicit value of the constant $\rho(k,M)$ can be found at the end of the proof. A closer look at the involved constants shows that $\rho(k,cM)=c\rho(k,M)$, i.e., $\rho$ scales as the map $T^2$ does.
}
\end{remark}

\begin{proof}
A $k$-periodic orbit $z$ of $T$ is also a $k$-periodic orbit of $T^2$. Let us denote by
$$
z:=p_1,p_2,\ldots,p_k:\ p_{i+1}=T^2(p_i),\ i=1,\ldots,k-1,\ T^2(p_k)=p_1
$$
the entire orbit of $T^2$. Recall that $C_1$ is the diameter of $M$. Then $|T^2(x)-x|\le 2C_1$ for every point $x$ by the definition of $T$, see Figure \ref{far}. In particular, $|p_{i+1}-p_i| \le 2C_1$, and hence, by the triangle inequality, 
\begin{equation} \label{eq:zp}
|z-p_i|\le 2(k-1)C_1,\ i=1,\ldots,k-1. 
\end{equation}
We want to use the estimate (\ref{eq:TV}) for every $p_i$. For that, one needs $|p_i| \ge \frac{1}{\delta}$ to hold for all $i=1,\ldots,k-1$. In view of (\ref{eq:zp}) and the triangle inequality, this will hold as long as $|z| \ge \frac{1}{\delta} + 2(k-1)C_1$. Therefore, we assume from now on that
$$
|z| \ge \frac{1}{\delta} + 2(k-1)C_1.
$$
Then  \eqref{eq:TV}, applied to $p_i$, gives
$$
|p_{i+1}-p_i-V(p_i)| \le \frac{6C}{|p_i|}
$$
for all $i$. It follows then from the telescope trick and triangle inequality that 
%
\begin{equation}\label{eq:VP}
\begin{aligned}
\left|\sum_{i=1}^k V(p_i)\right| &= \left|\sum_{i=1}^k (p_{i+1}-p_i-V(p_i))\right|\\
& \le \sum_{i=1}^k |p_{i+1}-p_i-V(p_i)| \\
&\le 6C \sum_{i=1}^k \frac{1}{|p_i|}.
\end{aligned}
\end{equation}
Recall the quantity $\eta$ from Lemma \ref{lm:2pts} and let us assume, in addition, that  
$$
|z| > \frac{2(k-1)C_1}{\eta}.
$$
If we combine this with \eqref{eq:zp}, we immediately obtain $|z-p_i| \le \eta |z|$ for all $i$. In particular, we can apply Lemma \ref{lm:2pts} and conclude that $\frac{|p_i|}{|z|}\geq\frac12$, and 
therefore 
$$
\sum_{i=1}^k \frac{1}{|p_i|} \leq \frac{2k}{|z|}.
$$
Combining this with \eqref{eq:VP} leads to
\begin{equation} \label{eq:VPC}
\left|\sum_{i=1}^k V(p_i)\right| \leq \frac{12kC}{|z|}.
\end{equation}
On the other hand, Lemma \ref{lm:2pts} asserts that $|V(z)-V(p_i)| < \frac{m}{2}$ for all $i$ where $\displaystyle m=\min_{x\in\R^{2d}\setminus\{0\}}|V(x)|$. In particular, $|V(z)| \ge m$, and it follows that
\begin{equation}\nonumber
\begin{aligned}
\left|\sum_{i=1}^k V(p_i)\right| &= \left|\sum_{i=1}^k [V(p_i)-V(z)] +kV(z)\right|\\ 
&\ge k|V(z)| - \sum_{i=1}^k |V(z)-V(p_i)| \\
&\ge \frac{km}{2}.
\end{aligned}
\end{equation}
Combined with (\ref{eq:VPC}), we have 
$$
\frac{12kC}{|z|} \geq \frac{km}{2}\ \ \ {\rm or}\ \ \ |z| \leq \frac{24C}{m}.
$$ 
In conclusion, we showed that
$$
|z| >\max\left\{\frac{1}{\delta} + 2(k-1)C_1,\frac{2(k-1)C_1}{\eta}\right\}\quad\Longrightarrow\quad |z| \leq \frac{24C}{m}
$$
for a $k$-periodic orbit $z$. In other words, if $z$ is a $k$-periodic orbit then
\begin{equation}\nonumber
|z|\leq\max \left\{\frac{1}{\delta} + 2(k-1)C_1,\ \frac{2(k-1)C_1}{\eta},\ \frac{24C}{m}\right\}=:\rho(k,M),
\end{equation}
as needed.
%
%
\proofend
\end{proof}

\begin{remark} 
{\rm Convexity plays a critical role in our considerations. We present a (rather extreme) example of what may happen otherwise.

Consider two oriented concentric circles. The respective outer billiard maps, $T_1$ and $T_2$, commute, hence every points is fixed by the map $T_2^{-1} T_1^{-1} T_2 T_1$. Properly oriented red arcs of these two circles in Figure \ref{4far}  can be combined into a non-convex (and self-intersecting) closed oriented curve, and the outer billiard map with respect to this curve will have arbitrarily distant 4-periodic orbits.
A similar construction can be made in higher dimensions.
\begin{figure}[hbtp] 
\centering
\includegraphics[width=5.5in]{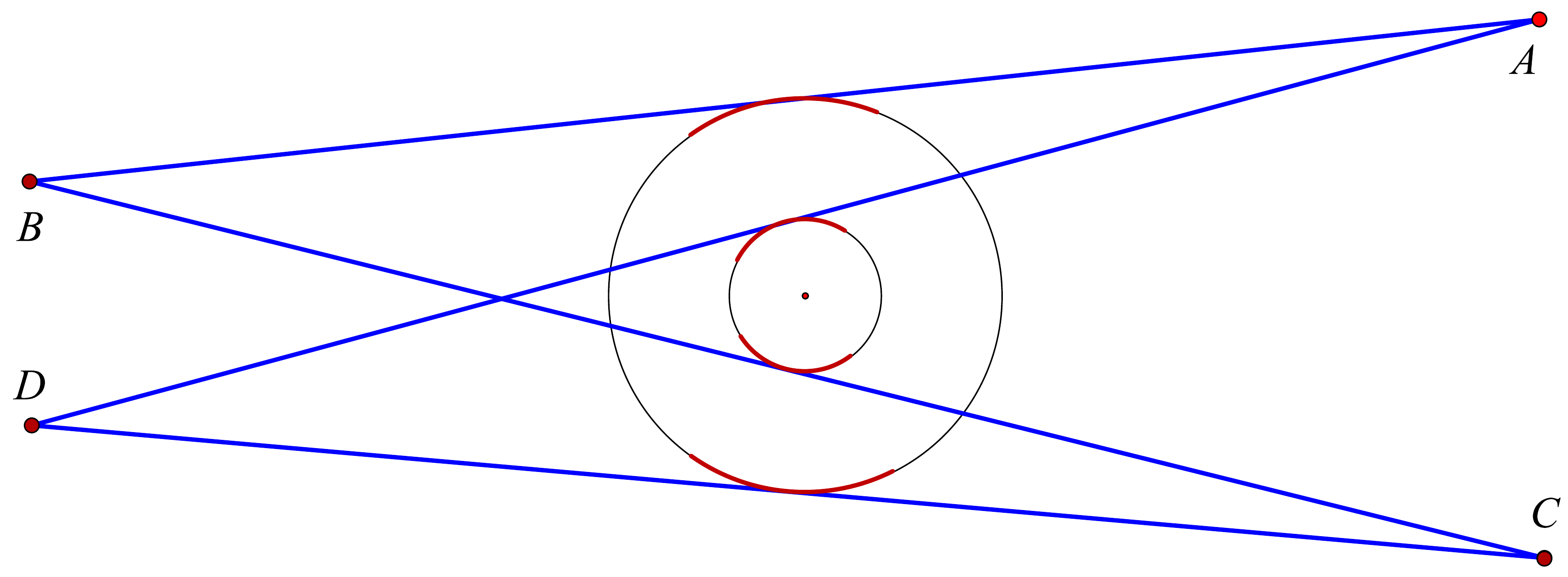}
\caption{A 4-periodic orbit $A\to B \to C \to D \to A$.}
\label{4far}
\end{figure}
}
\end{remark}
Finally, as promised, an elementary lemma. 

\begin{lemma}\label{micro_lemma}
Let $a,b\in \R^n$ satisfy, with respect to some norm $|\cdot|$ and constants $C,\Lambda>0$, the estimate
\begin{equation}\nonumber
|b|\geq\frac1\Lambda \quad\Longrightarrow\quad |a-b|\leq \frac{C}{|b|}\leq C\Lambda.
\end{equation} 
Then 
\begin{equation}\nonumber
|b|\geq\frac1\Lambda \quad\Longrightarrow\quad ||a|^2-|b|^2|\leq 2C+3C^2\Lambda^2
\end{equation}
holds.
\end{lemma}

\begin{proof}
From 
\begin{equation}\nonumber
|a|^2\leq(|a-b|+|b|)^2=|a-b|^2+|b|^2+2|a-b||b|
\end{equation}
we estimate for $|b|\geq\frac1\Lambda$
\begin{equation}\nonumber
\begin{aligned}
|a|^2-|b|^2&\leq \underbrace{|a-b|^2}_{\leq C^2\Lambda^2}+2\underbrace{|a-b|}_{\leq\frac{C}{|b|}}|b|\\
&\leq C^2\Lambda^2+2C.
\end{aligned}
\end{equation}
Similarly,
\begin{equation}\nonumber
\begin{aligned}
|b|^2\leq(|b-a|+|a|)^2&=|b-a|^2+|a|^2+2|a-b|\underbrace{|a|}_{\leq |b-a|+|b|}\\
&\leq |a-b|^2+|a|^2+2|a-b|^2+2|a-b||b|\\
&= |a|^2+3|a-b|^2+2|a-b||b|\\
&\leq |a|^2+3C^2\Lambda^2+2C,
\end{aligned}
\end{equation}
and we arrive at 
\begin{equation}\nonumber
-C^2\Lambda^2-2C\leq |b|^2-|a|^2\leq3C^2\Lambda^2+2C.
\end{equation}
This proves the claimed estimates.
\proofend 
\end{proof}

\end{document}